\documentclass[a4paper]{article}

\usepackage[english]{babel}
\usepackage[utf8x]{inputenc}
\usepackage[T1]{fontenc} 
\usepackage{scalerel}

\usepackage[a4paper,top=3cm,bottom=2cm,left=3cm,right=3cm,marginparwidth=1.75cm]{geometry}

\usepackage{amsmath}
\usepackage{mathtools}
\usepackage{amssymb}
\usepackage{amsthm}
\usepackage{natbib}
\usepackage{comment}
\usepackage{enumitem} 
\usepackage{subcaption}
\usepackage{cases}
\usepackage{graphicx}
\usepackage{wrapfig}
\usepackage[colorinlistoftodos]{todonotes}
\usepackage[colorlinks=true, allcolors=blue]{hyperref}
\usepackage{empheq} 
\usepackage[most]{tcolorbox} 
\usepackage[linesnumbered,ruled]{algorithm2e}
\usepackage{algorithmic}
\SetKwRepeat{Do}{do}{while}%
\usepackage{multirow}
\usepackage{bm}
\usepackage{siunitx}
\usepackage{footnote}
\makesavenoteenv{tabular}
\makesavenoteenv{table}
\usepackage{breakcites}

\usepackage{xcolor}
\usepackage{comment}
\usepackage[normalem]{ulem} 
\newcommand\str{\bgroup\markoverwith
{\textcolor{red}{\rule[0.5ex]{2pt}{1.5pt}}}\ULon} 
\usepackage{pifont}
\newcommand{\xmark}{\ding{55}}%
 
\newcommand{\al}{\alpha}

\newcommand{\tnabla}{\tilde{\nabla}}
\newcommand{\tPhi}{\tilde{\Phi}}

\newcommand{\id}{I_d}

\newcommand{\bigO}{\mathcal{O}}

\newcommand{\reals}{{\mathbb{R}}}

\newcommand{\E}{\mathbb{E}}

\newcommand{\err}{\text{err}}

\newtheorem{theorem}{Theorem}[section]  
\newtheorem{definition}[theorem]{Definition}
\newtheorem{example}[theorem]{Example}

\newtheorem{lemma}[theorem]{Lemma}
\newtheorem{remark}[theorem]{Remark}
\newtheorem{corollary}[theorem]{Corollary}

\newtheorem{assumption}[theorem]{Assumption}

\newcommand{\beq}{\begin{equation}}
\newcommand{\eeq}{\end{equation}}
\newcommand{\beqa}{\begin{eqnarray}}
\newcommand{\eeqa}{\end{eqnarray}}
\newcommand{\beqs}{\begin{equation*}}
\newcommand{\eeqs}{\end{equation*}}
\newcommand{\beqas}{\begin{eqnarray*}}
\newcommand{\eeqas}{\end{eqnarray*}}
\DeclarePairedDelimiter\ceil{\lceil}{\rceil}

\begin{document}
\title{An Optimal Multistage Stochastic Gradient Method for Minimax Problems}
\author{Alireza Fallah\thanks{Authors are in alphabetical order.\newline Department of Electrical Engineering and Computer Science, Massachusetts Institute of Technology, Cambridge, MA, USA. \{afallah, asuman, sarathp\}@mit.edu.} , Asuman Ozdaglar$^*$, Sarath Pattathil$^*$}
\date{}

\maketitle

\begin{abstract}
In this paper, we study the minimax optimization problem in the smooth and strongly convex-strongly concave setting when we have access to noisy estimates of gradients. In particular, we first analyze the stochastic Gradient Descent Ascent (GDA) method with constant stepsize, and show that it converges to a neighborhood of the solution of the minimax problem. We further provide tight bounds on the convergence rate and the size of this neighborhood. Next, we propose a multistage variant of stochastic GDA (M-GDA) that runs in multiple stages with a particular learning rate decay schedule and converges to the exact solution of the minimax problem. We show M-GDA achieves the lower bounds in terms of noise dependence without any assumptions on the knowledge of noise characteristics.  We also show that M-GDA obtains a linear decay rate with respect to the error's dependence on the initial error, although the dependence on condition number is suboptimal. In order to improve this dependence, we apply the multistage machinery to the stochastic Optimistic Gradient Descent Ascent (OGDA) algorithm and propose the M-OGDA algorithm which also achieves the optimal linear decay rate with respect to the initial error. To the best of our knowledge, this method is the first to simultaneously achieve the best dependence on noise characteristic as well as the initial error and condition number. 
\end{abstract}

\section{Introduction}
The minimax optimization problem has recently gained tremendous attention {as the canonical problem formulation for robust training of machine learning models and Generative Adversarial Networks (GANs) (see \cite{DBLP:conf/iclr/MadryMSTV18,goodfellow2014generative, pmlr-v70-arjovsky17a}).}
While many papers have studied the convergence of a broad range of algorithms in the deterministic setting, i.e., when the gradient information is exact, many aspects of this problem in the stochastic setting are yet to be explored. This is the main goal of our manuscript as we provide a framework for analyzing minimax optimization algorithms which can be used for both the deterministic and stochastic settings.

We consider the minimax problem 
\begin{equation}\label{main_prob}
\min_{x \in \mathbb{R}^m} \max_{y \in \mathbb{R}^n} f(x,y)    
\end{equation}
where $f : \mathbb{R}^m \times \mathbb{R}^n \to \mathbb{R}$ is $L-$smooth and $\mu$-strongly convex-strongly concave (See Section \ref{prelim} for the precise statement of our assumptions). The condition number of the problem is defined as $\kappa:= {L}/{\mu}$. Due to the strong convexity-strong concavity of the function $f$, this problem has a unique saddle point which we denote by $(x^*, y^*)$, i.e.,
\begin{align*}
f(x^*, y) \leq f(x^*, y^*) \leq f(x, y^*) \quad \forall~x \in \reals^n, y \in \reals^m. 
\end{align*}
In this paper, our main focus is on the case when the exact gradient information is not available and we only have access to an unbiased estimate through a stochastic oracle. 
More formally, we assume\footnote{see Assumption \ref{var_assmp:main} for detailed statement of the assumption} that at a point $(x_k,y_k)$, we have access to noisy estimates of the gradients  
 $\tnabla_x f(x_k,y_k)$ and $\tnabla_y f(x_k,y_k)$, which are unbiased estimates of $\nabla_x f(x_k,y_k)$ and $\nabla_y f(x_k,y_k)$, respectively, and their variances are bounded by $\sigma^2$.

This setting arises in many applications, including the problem of training GANs where the generator and the discriminator approximate the gradient by taking a batch of data $\mathcal{D}$ and computing $\frac{1}{|\mathcal{D}|} \sum_{\theta \in \mathcal{D}} \tilde{\nabla}_i f(x,y,;\theta)$ where $i \in \{x,y\}$ and $\tilde{\nabla}_i f(x,y,;\theta)$ is the gradient computed over a single data point $\theta$.
It is worth noting that the inexact gradient issue also appears in other scenarios such as {privacy-related applications where the noise is added intentionally to prevent the model from remembering possibly sensitive data and preserve privacy \cite{xie2018differentially}} or when the presence of noise is inevitable due to imperfections in communication and sensing.


In solving the minimization problem in the stochastic setting, it is well-known that, for many algorithms, the squared distance of the iterates to the solution of the minimization problem can be bounded by the sum of two terms: \textit{bias} and \textit{variance} \cite{bach-non-strongly-cvx, ghadimi2012optimal, MASG}. The bias term captures the effect of the initialization expressed in terms of the distance of the initial point to the solution, and is independent of the noise parameters. The variance term depends on noise characteristics ($\sigma^2$ in our case) and is independent of the initialization error. For the minimization problem with strongly convex objective function, and in the noiseless case (with only the bias term), \cite{nemirovsky1983problem} have shown the lower bound of $\Theta \left ( \exp(-\Theta(1) n/\sqrt{\kappa}) \right )$ for the distance of the $n$-th iterate to the optimal solution. With noise \cite{raginsky2011information} have shown the lower bound increases to $\Theta(\sigma^2/n)$.
Several papers have highlighted the trade-off between bias and variance which arises in design of optimization algorithms \cite{RAGM} and tried to achieve both lower bounds simultaneously \cite{ghadimi2013optimal, MASG}.  

In this paper, we highlight this bias-variance decomposition in evaluating the performance of algorithms that solve the minimax problem. For the bias term, i.e., the deterministic case, \cite{ibrahim2019lower} have recently shown the lower bound $\bigO(1)\exp(-\Theta(1) n/\kappa)$ highlighting that the dependence on condition number increases from $\sqrt{\kappa}$ in minimization problems to $\kappa$ for minimax problems.  
For the variance term, since the minimax problem is a special case of the minimization problem, the lower bound $\bigO(\sigma^2/n)$ of the minimization problem is also valid for the minimax problem.
While this lower bound for variance term has been obtained \cite{hsieh2019convergence, rosasco2014stochastic} at the cost of making the bias term sublinear, the question of whether a linear rate $\bigO(1)\exp(-\Theta(1) n/\kappa)$ in bias and $\bigO(\sigma^2/n)$ in variance could be achieved simultaneously has not been addressed prior to this work.

In what follows, we first provide a summary of related works and then discuss the main contributions of our paper.
\subsection{Related Work}

\subsubsection{Deterministic Case}
Many papers have studied the minimax problem when the exact gradient information is available. In the case of Gradient Descent Ascent (GDA) method, \cite{DuHu} analyzes its performance for the special case of bilinear coupling, i.e., when $f(x, y) = g(x) + y^{\top} A x - h(y)$ where $g$ is smooth and convex, $h$ is smooth and strongly convex, and the matrix $A$ has full column rank. They show that running the GDA algorithm for $n$ steps on this problem reaches a point which is $\mathcal{O}(1)\left(1/n^2 \right)$ close to the saddle point. In addition, when the function $g(\cdot)$ is assumed to be strongly convex, GDA reaches a point which is $\mathcal{O}(1)\text{exp}\left(-  \Theta(1) n/\kappa^2 \right)$ close to the saddle point after $n$ steps.       
\cite{LiangStokes} extend this result to a general function $f(x,y)$ which is strongly convex in $x$ and strongly concave in $y$ {(achieving the same rate of convergence as \cite{DuHu})}. 
Several other gradient based algorithms like the Optimistic Gradient Descent Ascent (OGDA) method (see \cite{DBLP:conf/iclr/DaskalakisISZ18}) and the Extragradient method \cite{korpelevich1976extragradient} have been analyzed in recent papers including 
\cite{mokhtari2019unified, LiangStokes, GidelBVVL19, mokhtari2019proximal, hsieh2019convergence}. These papers analyze these algorithms in several settings including bilinear, strongly convex-strongly concave and convex-concave. More specifically, \cite{GidelBVVL19,mokhtari2019unified} show that when the objective function is strongly convex-strongly concave, running the OGDA and Extragradient algorithms for $n$ steps reaches a point which is $\bigO(1)\exp(-\Theta(1) n/\kappa)$ close to the saddle point.  

\subsubsection{Stochastic Case}

The papers which are closest to our results are \cite{rosasco2014stochastic} and \cite{hsieh2019convergence}. \cite{rosasco2014stochastic} propose a forward-backward splitting algorithm to solve the stochastic minimax problem (they solve the more general problem of monotone inclusions). When the function is strongly convex-strongly concave, they show convergence at a rate of $\mathcal{O}(\|z_0 - z^*\|^2/n^p + \sigma^2 c^p/{n})$ to the saddle point, where $p$ is any constant greater than $0$ and $c$ is a constant larger than 1.
\cite{hsieh2019convergence} show that the stochastic version of OGDA converges to the saddle point at a rate of $\mathcal{O}(\frac{1}{n})$ for both bias and variance when the objective function is strongly convex-strongly concave.

There are several papers which analyze the stochastic minimax problem when the objective function is convex-concave. \cite{juditsky2011solving} propose the stochastic mirror-prox algorithm (a special case of which is the stochastic extragradient method) to solve the convex-concave saddle point problem with noisy gradients. They assume the constraint set is compact and show a convergence rate of $\mathcal{O}(1/\sqrt{n})$ (the result in this paper improves on the robust stochastic approximation algorithm proposed in \cite{nemirovski2009robust}). 
\cite{chen2014optimal} proposes an accelerated primal dual algorithm which achieves a convergence rate of $\mathcal{O}(1/\sqrt{n})$.
Recently, \cite{mertikopoulos2018mirror} analyzed the stochastic extragradient algorithm for coherent minimax problems (a condition slightly weaker than convex-concave assumption) and they show asymptotic convergence to a saddle point.
\cite{GidelBVVL19} analyzed a single call version of extragradient (which corresponds to OGDA) when the function is convex-concave and they showed that in the stochastic setting, this algorithm converges to the saddle point at a rate of $\mathcal{O}(1/\sqrt{n})$. 

Another line of work is the case where the objective function has a finite sum structure and the gradient of the entire function cannot be computed at each step. Several papers including \cite{bot2019forward, PalaBach, chavdarova2019reducing, iusem2017extragradient} analyze this setting and apply variance reduction techniques (like SVRG and SAGA) to improve convergence rates to the saddle point. 
\subsection{Our Contribution}
{We first analyze GDA with constant stepsize (learning rate) where we build our analysis by casting it as a dynamical system, an approach that has gained attention in the optimization and machine learning literature recently \cite{lessard2016analysis, hu2017dissipativity, RAGM, MASG}. In particular, we show that GDA with any stepsize $\alpha \leq \mu/(4 L^2)$ converges to an $\bigO(\alpha)$ neighborhood of the optimal solution at a linear rate $\exp(- \alpha \mu k)$. Next, we propose a novel Multistage-Stochastic Gradient Descent Ascent scheme (inspired from \cite{MASG}) which achieves a rate of $\mathcal{O}({\sigma^2}/{n})$ for the variance term (which is optimal in terms of $n$ dependence) and a rate of $\bigO(1)\exp(-\Theta(1) n/\kappa^2)$ for the bias term, and we show that the $n$ and $\kappa$ dependence of the latter cannot be improved for GDA dynamics.}

{Next, we focus on the OGDA method which has gained widespread attention for solving minimax problems. We first highlight that OGDA also converges to an $\bigO(\alpha)$ neighborhood of the optimal solution with linear rate $\exp(- \alpha \mu k)$, but allows for a broader range of $\alpha \leq 1/(8 L)$ for the stepsize. Then, we introduce the Multistage version of Stochastic Optimistic Gradient Descent Ascent (M-OGDA) which achieves the rate of $\mathcal{O}({\sigma^2}/{n})$ for the variance term and a rate of $\bigO(1)\exp(-\Theta(1) n/\kappa)$ for the bias term which improves on the $\bigO(1)\exp(-\Theta(1) n/\kappa^2)$ decay of the bias term of GDA and matches the lower bound shown in \cite{ibrahim2019lower}.
}

\begin{table}[t]
\small
\caption{Summary of results (up to $\bigO(1)$ constants)}
\label{table-1}
\vskip 0.1in
\begin{center}
\begin{tabular}{ |l|c|c|c|}
\hline
\multirow{2}{*} {\textbf{Algorithm}}& \multirow{2}{*}{\textbf{Bias}} & \multirow{2}{*}{\textbf{Var.}} &\textbf{Extra} \\
& & & \textbf{Info.}\\
\hline
\textbf{OGDA} &  \multirow{2}{*}{$1/n$} &  \multirow{2}{*}{$\sigma^2/n$} & \multirow{2}{*}{\xmark} \\
\cite{hsieh2019convergence} & & & \\
\hline
\textbf{Forward-back.Splitting} &  \multirow{2}{*}{$1/n^p$} &  $c^p \sigma^2/n$ & \multirow{2}{*}{\xmark} \\
\cite{rosasco2014stochastic} & & \small{$(c>1)$} & \\
\hline
\textbf{M-GDA \& M-OGDA} & \multirow{2}{*}{$1/n^p$}  & \multirow{2}{*}{$p \sigma^2/n$} & \multirow{2}{*}{\xmark} \\
Corollaries \textbf{\ref{result_general} \& \ref{M-OGDA_AllResults}} & & & \\
\hline
\textbf{M-GDA} & \multirow{2}{*}{$\exp(-n/\Theta(\kappa^2))$} & \multirow{2}{*}{$\sigma^2/n$} & \multirow{2}{*}{$n$}\\
Corollary \textbf{\ref{result_knowing_n}} & & & \\
\hline
\textbf{M-OGDA} & \multirow{2}{*}{$\exp(-n/\Theta(\kappa))$}  & \multirow{2}{*}{$\sigma^2/n$} & \multirow{2}{*}{$n$} \\
Corollary \textbf{\ref{M-OGDA_AllResults}} & & & \\
\hline
\end{tabular}
\end{center}
\end{table}%
\subsection{Notation}
We denote $d\times d$ identity and zero matrices by $\id$ and $0_d$, respectively. Throughout this paper, all vectors are represented as column vectors. The superscript $^\top$ represents the transpose of a vector or a matrix. For two matrices $A \in \mathbb{R}^{m \times n}$ and $B \in \mathbb{R}^{p \times q}$, their Kronecker product is represented by $A\otimes B$. Also, $A \succeq B$ implies that $A - B$ is a symmetric and positive semidefinite matrix.

\section{Preliminaries}\label{prelim}
We first state formally the strong convexity(concavity) and smoothness properties of a function. 
\begin{definition} \label{def:lips_grad}
A convex function $\phi: \reals^n\to\reals$ is $L$-smooth and $\mu$-strongly convex if it satisfies the following two conditions for all $x,\hat{x} \in \mathbb{R}^n$:
\begin{align}
&\Vert \nabla \phi(x) - \nabla \phi(\hat{x}) \Vert \leq L \Vert x - \hat{x} \Vert \label{smooth}\\
&\phi(x)-\phi(\hat{x})-\nabla \phi(\hat{x})^\top (x-\hat{x}) \geq \frac{\mu}{2} \Vert x-\hat{x} \Vert^2 \label{SConvex}
\end{align}
Further, $\phi(x)$ is $\mu$-strongly concave if $-\phi(x)$ is $\mu$-strongly convex.
\end{definition} 
For an $L-$smooth convex function $\phi(\cdot)$, we have the following characterization (see Theorem 2.1.5 in \cite{nesterov_convex}):
\begin{align}
    \phi(x) \leq \phi(\hat{x}) + \nabla \phi(\hat{x})^\top (x - \hat{x}) + \frac{L}{2}\| x - \hat{x} \|^2
    \label{eq:l_smooth_ch}
\end{align}

Throughout the paper, we assume the following: 
\begin{assumption}\label{var_assmp:main}
We assume at iterate $(x,y)$, we have access to $\tilde{\nabla} f_x (x, y, \zeta), \tilde{\nabla} f_y (x, y, \xi)$ which are unbiased estimates of $\nabla f_x (x, y)$ and $\nabla f_y (x, y)$, respectively, i.e.,
\begin{equation}\label{var_assmp:a}
\begin{aligned}
\mathbb{E}\left[\tilde{\nabla} f_x \left(x, y, \zeta\right)\Big|(x,y)\right]= \nabla f_x \left(x,y\right),\\
\mathbb{E}\left[\tilde{\nabla} f_y \left(x, y, \xi \right)\Big|(x,y)\right]= \nabla f_y \left(x,y\right). 
\end{aligned}
\end{equation}
In addition, we assume $\zeta$ and $\xi$ are independent from each other and previous iterates. Moreover, we assume
\begin{equation}\label{var_assmp:b}
\begin{aligned}
\mathbb{E} & \left[\left\Vert \tilde{\nabla} f_x \left(x,y,\zeta \right) - \nabla f_x \left(x,y\right) \right\Vert^{2}\Big|(x,y)\right]\leq\sigma^{2}, \\
\mathbb{E} & \left[\left\Vert \tilde{\nabla} f_y \left(x,y,\xi \right) - \nabla f_y \left(x,y\right) \right\Vert^{2}\Big|(x,y)\right]\leq\sigma^{2}.
\end{aligned}
\end{equation}	
To simplify the notation, we suppress the $\zeta$ and $\xi$ dependence throughout the paper.
\end{assumption}
\begin{assumption}
\label{ass:scsc}
The function $f(x, y)$ is continuously differentiable in $x$ and $y$. For any $\hat{y} \in \reals^n$, $f(x, \hat{y})$ is $L_x$-smooth and $\mu_x$-strongly convex as a function of $x$. Similarly, for any $\hat{x} \in \reals^m$, $f(\hat{x}, y)$ is $L_y$-smooth and $\mu_y$-strongly concave as a function of $y$.

\noindent In addition, the gradient $\nabla_{x}f(x, y)$ is $L_{xy}$-Lipschitz in $y$, i.e., 
\begin{align*}
\| \nabla_{x}f(x, y) - \nabla_{x}f(x, \hat{y})\| &\leq L_{xy}\|y - \hat{y} \| \quad {\forall} \: x \in \reals^m.
\end{align*}
Similarly, the gradient $\nabla_{y}f(x, y)$ is $L_{yx}$-Lipschitz in $x$, i.e., \begin{align*}
\| \nabla_{y}f(x, y) - \nabla_{y}f(\hat{x}, y)\| &\leq L_{yx}\|x - \hat{x} \| \quad {\forall} \: y \in \reals^n.
\end{align*}
\end{assumption}

 Note that this assumption leads to the saddle point $(x^*, y^*)$ being unique and, in addition, we have $\nabla_x f (x^*, y^*) = 0$ and $\nabla_y f (x^*, y^*) = 0$.
\begin{remark}\label{cond_numb_remark}
Under Assumptions \ref{ass:scsc},  we call the function $f(\cdot, \cdot)$ as $L$-smooth and $\mu$-strongly convex- strongly concave where $\mu = \min \{ \mu_x, \mu_y \}$ and $L = \max \{ L_x, L_y, L_{xy}, L_{yx} \}$. We define the condition number of the problem as $\kappa \triangleq L/\mu$.
\end{remark}

We next present some key properties of smooth strongly convex-strongly concave functions that will be used in our analysis. Define:
\begin{align}\label{Phi_def}
    \Phi(z) \triangleq \begin{bmatrix} \nabla_{x} f(x, y) \\ - \nabla_{y} f(x, y) \end{bmatrix}
\end{align}
where $z = (x^\top, y^\top)^\top$ $\in \reals^{m+n}$. For $z = (x^\top, y^\top)^\top, \hat{z} = (\hat{x}^\top, \hat{y}^\top)^\top$, we define 
\begin{align}
    \|z - \hat{z} \|^2 \triangleq \|x - \hat{x} \|^2 + \| y - \hat{y} \|^2.
\end{align}
Also, we define $z^* \triangleq ({x^*}^\top, {y^*}^\top)^\top$ as the unique saddle point. The following lemma follows from the strong convexity and smoothness properties of $f$.
\begin{lemma}
\label{lemma:SCSC_monotone}
Let $z, \hat{z} \in \reals^{m+n}$. Recall the definition of $\Phi$ from \eqref{Phi_def} and $\mu$ and $L$ from Assumption \ref{ass:scsc} and Remark \ref{cond_numb_remark}. Then,
\begin{align}\label{sc_sm_g}
    L \| z - \hat{z} \|^2 \geq \langle \Phi(z)- \Phi(\hat{z}), z - \hat{z} \rangle \: \geq \: \mu \| z - \hat{z} \|^2.
\end{align}
\end{lemma}
\begin{proof}
Check Appendix \ref{lemma:SCSC_monotone_proof}.
\end{proof}
Using Lemma \ref{lemma:SCSC_monotone}, we can prove the following result 
\begin{lemma}\label{lem_Phi}
Let $z, \hat{z} \in \reals^{m+n}$. Recall the definition of $\Phi$ from \eqref{Phi_def} and $\mu$ and $L$ from Assumption \ref{ass:scsc} and Remark \ref{cond_numb_remark}. Then,
\begin{align*}
    \langle \Phi(z)- \Phi(\hat{z}), z - \hat{z} \rangle \geq \frac{\mu}{4L^2}\| \Phi(z) - \Phi(\hat{z}) \|^2
\end{align*}
\end{lemma}
\begin{proof}
Check Appendix \ref{lem_Phi_proof}.
\end{proof}
Using Lemmas \ref{lemma:SCSC_monotone} and \ref{lem_Phi}, we immediately obtain the following result which we state in the form of a matrix inequality since this form is more convenient for subsequent analysis.
\begin{corollary}\label{corollary:Matrix_format}
Let $z \in \reals^{m+n}$. Recall the definition of $\Phi$ from \eqref{Phi_def} and $\mu$ and $L$ from Assumption \ref{ass:scsc} and Remark \ref{cond_numb_remark}. Then,
\begin{align}
    \begin{bmatrix} z-z^* \\ \Phi(z) \end{bmatrix}^{\top}
    \begin{bmatrix}
     \mu & -1 \\ -1 & \mu/(4L^2)    
    \end{bmatrix}
    \begin{bmatrix} z-z^* \\ \Phi(z) \end{bmatrix} \leq 0.
\end{align}
\end{corollary}


\section{Analysis of Stochastic Gradient Descent Ascent Method}
In this section, we study the Stochastic Gradient Descent Ascent (GDA) algorithm, which is given by:
\begin{subequations}\label{GDA_update:main}
\begin{align}
    x_{k+1} &= x_k - \alpha \tnabla_x f(x_k,y_k) \label{GDA_update:a} \\
    y_{k+1} &= y_k + \alpha \tnabla_y f(x_k,y_k). \label{GDA_update:b}
\end{align}
\end{subequations}
This can be succinctly written as:
\begin{align}\label{GDA_update:z}
    z_{k+1} &= z_k - \alpha \tPhi(z_k) 
\end{align}
where $z_{k} = (x_k^\top, y_k^\top)^\top$ and 
\begin{align}\label{stoch_Phi}
    \tPhi(z_k) \triangleq \begin{bmatrix} \tnabla_{x} f(x_k, y_k) \\ - \tnabla_{y} f(x_k, y_k) \end{bmatrix}.
\end{align}
Using this notation, we can represent GDA as a dynamical system as follows:
\begin{align}\label{dynamical_GDA}
z_{k+1} = A z_k + B \tPhi(z_k),
\end{align}
where $A = I_{m+n}, B = -\alpha I_{m+n}$. We study the convergence properties of the sequence $\{z_k\}_k$ through the evolution of the Lyapunov function $V_p(z) = (z-z^*)^\top P (z-z^*)$ where $P = p \otimes I_{m+n}$ with $p \geq 0$ an arbitrary constant. In particular, in the following lemma, we first bound the difference of $\E[V_p(z_{k+1})] - \rho^2 \E[V_p(z_{k})]$ for any $\rho \geq 0$ and $k \geq 0$. We skip the proof as it is very similar to the proof of Lemma B.1 in \cite{MASG}.
\begin{lemma} \label{storage_update}
Let $P = p \otimes I_{m+n}$ with $p \geq 0$ and consider the function $V_p(z)= (z-z^*)^\top P  (z-z^*)$. Then we have
\begin{align}
\label{storage_update_Expec}
&\E[V_p(z_{k+1})] - \rho^2 \E[V_p(z_{k})] \leq 2 \sigma^2 \alpha^2 p + \E\left[\begin{bmatrix} 
	z_k-z^*\\
    \Phi (z_k)
\end{bmatrix}^\top
\begin{bmatrix} 
	A^\top P A - \rho^2 P & A^\top P B\\
    B^\top P A & B^\top P B
\end{bmatrix}
\begin{bmatrix} 
	z_k-z^*\\
    \Phi (z_k)
\end{bmatrix}\right].
\end{align}
\end{lemma}
Next, using this lemma, we characterize the convergence of GDA.
\begin{theorem}\label{Thm_GDA} 
Suppose that the conditions in Assumptions \ref{var_assmp:main} and \ref{ass:scsc} are satisfied. Let $\{z_k\}$ be the iterates generated by GDA \eqref{GDA_update:z} with $0 < \alpha \leq \frac{\mu}{4L^2}$. Then, for any $k \geq 1$, we have
\begin{equation}\label{ineq_GDA_1step}
\E[\|z_{k}-z^*\|^2] \leq  (1- \alpha \mu) \E[\|z_{k-1}-z^*\|^2] + 2 \alpha^2 \sigma^2.
\end{equation}
As a result, the error of GDA after $k$ steps is bounded by
\begin{equation}\label{ineq_GDA_kstep}
\E[\|z_{k}-z^*\|^2] \leq  (1- \alpha \mu)^{k} \E[\|z_{0}-z^*\|^2] + 2 \alpha \sigma^2/\mu.
\end{equation}
\end{theorem}
\begin{proof}
See Appendix \ref{Thm_GDA_proof}.
\end{proof}
\begin{remark}
It is worth noting that the range for the stepsize $\alpha$ in Theorem \ref{Thm_GDA} is upper bounded by $\mu/4L^2$ (as opposed to just a function of the Lipschitz parameter $L$, as is the case for Gradient Descent in minimization problems). This is consistent with the fact that GDA may diverge when the strong convexity parameter $\mu$ is 0, i.e., the function is convex-concave (see the Bilinear example in \cite{DBLP:conf/iclr/DaskalakisISZ18}).
\end{remark}
\subsection{Tightness of the Results}
In this subsection, we give an example of a function where after running GDA
for $n$ iterations reaches a point which is $\mathcal{O}(1)\text{exp}\left(- \Theta(1) n/\kappa^2 \right)$ close to the saddle point.
Consider the function 
\begin{align}
    f(x,y) = \frac{\mu}{2} x^2 + L xy - \frac{\mu}{2}y^2
    \label{eq:def_quad_fn}
\end{align}
where $x,y \in \mathbb{R}$ and $0 < \mu \leq L$.
The condition number of this function is $\kappa = \frac{L}{\mu}$ and the saddle point of this function is $(x,y) = (0,0)$.
{\begin{example} \label{lemma:quad_example}
Let $\{ x_k, y_k\}$ be the iterates generated by GDA for the objective function given in Equation \eqref{eq:def_quad_fn}. Then, \\
(i) if the gradient at each step is exactly available (i.e. the updates reduce to the deterministic GDA updates), we have:
\begin{align}
     \|x_{k+1}\|^2 + \|y_{k+1}\|^2  &\geq  \left(1 - \frac{1}{\kappa^2} \right)(\|x_{k}\|^2 + \|y_{k}\|^2 )
\end{align}
(ii) if at each step the gradients are corrupted by additive i.i.d. noise with a distribution $\mathcal{N}(0,\sigma^2)$, we have
\begin{align}
    \E[ \| & x_{k+1}\|^2  + \|y_{k+1}\|^2 ] \geq (1-2 \alpha \mu) \E[(\|x_{k}\|^2 + \|y_{k}\|^2)] + 2\alpha^2 \sigma^2
\end{align}
\end{example}}
\begin{proof}
See Appendix \ref{lemma:quad_example_proof}.
\end{proof}
{
Example \ref{lemma:quad_example}(i) shows that in order to find the saddle point of the function $f$ defined in equation \eqref{eq:def_quad_fn}, we need to run at least $\mathcal{O}(\kappa^2 \log(1/\epsilon))$ steps of GDA (i.e. the deterministic case) to reach a point which is $\epsilon$-close to the solution, showing that this dependence on $\kappa$ cannot be improved.
Example \ref{lemma:quad_example}(ii) shows that when the gradients are corrupted by noise with variance $\sigma^2$, GDA reaches an $\mathcal{O}(\alpha)$ neighborhood of the saddle point and this dependence on $\alpha$ cannot be improved.}

\section{A Multistage Stochastic Gradient Descent Ascent Method (M-GDA)}\label{sec:MGDA}
Our result in Theorem \ref{Thm_GDA} shows that for GDA with constant stepsize $\alpha$, the iterates converge to an $\bigO(\alpha)$ neighborhood of the saddle point. In this section, we introduce a new method which is a variant of GDA with progressively decreasing stepsize that converges to the exact unique saddle point of problem \ref{main_prob}. Our proposed algorithm, Multistage Stochastic Gradient Descent Ascent (M-GDA), which is presented in Algorithm \ref{Algorithm1}, runs in several stages where each stage is the GDA method with constant stepsize. In what follows, we show our multistage method with a {carefully chosen learning rate and step length evolution}
achieves linear decay in the bias term as well as optimal variance dependence without any knowledge of the noise properties.
\begin{algorithm}[tb]
    \caption{Multistage Stochastic Gradient Descent Ascent Algorithm (M-GDA)}
    \label{Algorithm1} 
    \begin{algorithmic}
    \STATE {\bfseries Input:}Initial points $x_0^0, y_0^0$, the stepsize sequence $\{\al_k\}_{k=1}^K$, the stage-length sequence $\{n_k\}_{k=1}^K$.
    \STATE Set $n_0 = 0$;
    \FOR{$k = 1;\ k \leq K;\ k = k + 1$}
    \STATE Set $x_0^k = x_{n_{k-1}}^{k-1}$ and $y_0^k = y_{n_{k-1}}^{k-1}$;
    \FOR{$m = 0;\ m < n_k;\ m = m + 1$}
    \STATE Set $x_{m+1}^k = x_m^k - \alpha_k \tilde{\nabla_x} f(x_{m}^k, y_{m}^k)$ 
    \STATE Set $y_{m+1}^k = y_m^k + \alpha_k \tilde{\nabla_y} f(x_{m}^k, y_{m}^k)$
    \ENDFOR
    \ENDFOR
    \end{algorithmic}
\end{algorithm}
\sloppy
\begin{theorem}\label{main_result}
Suppose that the conditions in Assumptions \ref{var_assmp:main} and \ref{ass:scsc} are satisfied. Let $\{ (x_m^k,y_m^k)_{m=0}^{n_k} \}_{k=1}^K$ be the iterates generated by M-GDA (Algorithm \ref{Algorithm1}) with the following parameters
\begin{equation}\label{M-GDA_parametres}
\begin{aligned}
\alpha_1 &= \frac{\mu}{4L^2}, n_1 \geq 1 \\
\alpha_k &= \frac{\mu}{L^2 2^{k+2}}, n_k = \ceil{p 2^{k+2} \kappa^2 \log(2)} \quad (k \geq 2),	
\end{aligned}
\end{equation}
where $p \geq 2$ is an arbitrary positive number. Then, for any $k \geq 1$, we have	
\begin{align*}
\E \left [ \| z_{n_k}^k - z^* \|^2 \right ] & \leq \frac{\exp\left (-n_1/(4 \kappa^2)\right) }{2^{p(k-1)}} \|z_0 - z^*\|^2 + \frac{\sigma^2}{2^k L^2}
\end{align*}
where $z_m^k = ((x_m^k)^\top, (y_m^k)^\top)^\top$ for any $0 \leq m \leq n_k$.
\end{theorem}
\begin{proof}
We prove the result by induction on $k$. To simplify the notation, we define $\err_k := \E\left [\| z_{n_k}^k - z^* \|^2 \right ]$. 

First, and for $k=1$, note that, by using Theorem \ref{Thm_GDA} along with the fact that $1- \alpha \mu \leq \exp(-\alpha \mu)$, we have:
\begin{align}
\err_1 & \leq \exp(- \alpha_1 n_1 \mu) \|z_0-z^*\|^2 + 2 \alpha_1 \sigma^2/\mu  = \exp(-\frac{n_1}{4\kappa^2}) \|z_0-z^*\|^2 + \frac{\sigma^2}{2L^2},   
\end{align}
where we plugged in $\alpha_1 = \frac{\mu}{4L^2}, n_1 \geq 1$ to obtain the last equality. Hence, the result holds for $k=1$. Now, assume the result holds for $k$, and we show it for $k+1$. Note that, Theorem \ref{Thm_GDA} for stage $k+1$ yields:
\begin{align}\label{ineq1_main_result}
\err_{k+1} & \leq \exp(- \alpha_{k+1} n_{k+1} \mu) \err_k + 2 \alpha_{k+1} \sigma^2/\mu \\
& \leq \exp(-p \log(2)) \err_k + \frac{\sigma^2}{2^{k+2}L^2}, 
\end{align}
where we used $\alpha_{k+1} = \mu/(L^2 2^{k+3})$ and $n_{k+1} \geq p 2^{k+3} \kappa^2 \log(2)$ to derive the last inequality. Now, note that, by induction hypothesis, we have
\begin{equation}
\err_k \leq \frac{\exp\left (-n_1/(4 \kappa^2)\right) }{2^{p(k-1)}} \|z_0 - z^*\|^2 + \frac{\sigma^2}{2^k L^2}.    
\end{equation}
Substituting this bound in \eqref{ineq1_main_result}, we obtain
\begin{align}
\err_{k+1} &\leq  \frac{\exp\left (-n_1/(4 \kappa^2)\right) }{2^{pk}} \|z_0 - z^*\|^2  + \frac{\sigma^2}{L^2} (\frac{1}{2^{k+2}} + \frac{1}{2^{k+p}}) \\
&\leq \frac{\exp\left (-n_1/(4 \kappa^2)\right) }{2^{pk}} \|z_0 - z^*\|^2 + \frac{\sigma^2}{2^{k+1}L^2}
\end{align}
where the last bound follows from $p \geq 2$. This completes the proof.
\end{proof}
The above theorem provides an upper bound on the distance of the last iterate of each stage to the saddle point of problem \eqref{main_prob}. Using this result, and in the following corollary, we provide an upper bound on the distance of any iterate from the saddle point. Before stating this corollary, let $\{z_n\}_n$ be the sequence which is obtained by concatenating the $\{ (x_m^k,y_m^k)_{m=0}^{n_k} \}_{k=1}^K$ sequences, i.e., 
\begin{equation}\label{concatenate}
z_n = \left ( \left (x_{n-\sum_{i=1}^{k-1} n_i}^{k} \right )^\top, \left (y_{n-\sum_{i=1}^{k-1} n_i}^{k} \right )^\top \right )^\top \quad \text{for } \sum_{i=1}^{k-1} n_i < n \leq \sum_{i=1}^{k} n_i.
\end{equation}
\begin{corollary}\label{cor_anyN}
Suppose that the conditions in Assumptions \ref{var_assmp:main} and \ref{ass:scsc} are satisfied. Consider running M-GDA (Algorithm \ref{Algorithm1}) with the parameters given in Theorem \ref{main_result}. Also, recall the definition of the concatenated sequence $\{ z_n \}$ from \eqref{concatenate}. Then, for any $n > n_1$, we have
\begin{align}
 \E \left [ \| z_n - z^* \|^2 \right]  & \leq \bigO(1) \left ( \frac{\exp\left (-n_1/(4 \kappa^2)\right) }{\left ((n-n_1)/ (\kappa^2 p) \right )^p} \|z_0 - z^*\|^2   + \frac{p\sigma^2}{(n-n_1) \mu^2} \right ). 
\end{align}
\end{corollary}
\begin{proof}
Check Appendix \ref{Proof_cor_anyN}	
\end{proof}
We interpret this result in two different regimes. First, we consider the case where we are given a fixed budget of $n$ iterations. In this case, the following corollary shows how we can tune the parameters to obtain linear decay in the bias term as well as $\bigO(1/n)$ reduction in the variance term. We omit the proof as it is an immediate application of Corollary \ref{cor_anyN}.
\begin{corollary}\label{result_knowing_n}
Suppose that the conditions in Assumptions \ref{var_assmp:main} and \ref{ass:scsc} are satisfied. Consider running M-GDA (Algorithm \ref{Algorithm1}) with the parameters given in Theorem \ref{main_result} and $p=2, n_1= \frac{n}{C}$ with $C \geq 2$.  Also, recall the definition of the concatenated sequence $\{ z_n \}$ from \eqref{concatenate}. Then, for any $n \geq 2 \kappa^2$, we have
\begin{align}
 \E \left [ \| z_n - z^* \|^2 \right]  & \leq \bigO(1) \left ( \exp\left (-\Theta(1) n/\kappa^2\right)\|z_0 - z^*\|^2   + \frac{\sigma^2}{n \mu^2} \right ). 
\end{align}
\end{corollary}
Finally, in the following corollary, we illustrate how our results can be applied to the case where we do not know the number of iterations in advance.
\begin{corollary}\label{result_general}
Suppose that the conditions in Assumptions \ref{var_assmp:main} and \ref{ass:scsc} are satisfied. Consider running M-GDA (Algorithm \ref{Algorithm1}) with the parameters given in Theorem \ref{main_result} and $n_1 = \ceil{4 p \kappa^2 \log(p \kappa^2)}$ for an arbitrary $p \geq 2$.  Also, recall the definition of the concatenated sequence $\{ z_n \}$ from \eqref{concatenate}. Then, for any $n \geq 2 n_1$, we have
\begin{align}
 \E \left [ \| z_n - z^* \|^2 \right]  & \leq \bigO(1) \left ( 1/n^p \|z_0 - z^*\|^2 + \frac{p\sigma^2}{n \mu^2} \right ). 
\end{align}
\end{corollary}

It is worth noting that the results in Corollaries \ref{result_knowing_n} and \ref{result_general} are presented in terms of the $n^{th}$ iterate $z_n$ which is obtained by concatenating the iterates of all stages, including inner iterations (as given in \eqref{concatenate}). In fact, while it is true that the number of inner stage iterations increases, the bounds in Table \ref{table-1} and Corollaries \ref{result_knowing_n} and \ref{result_general}, are all based on the total number of iterations, and therefore, they take into account the inner stage iterations. 
\section{A Multistage Stochastic Optimistic Gradient Descent Ascent Method (M-OGDA)}\label{sec:MOGDA}
As we showed in previous Section, M-GDA achieves the optimal variance rate $\bigO(\sigma^2/n)$ as well as linear decay $\bigO(1)\exp(-\Theta(1) n/\kappa^2)$ in the bias term. However, the dependence of the latter to condition number $\kappa$ is suboptimal compared to the lower bound presented in \cite{ibrahim2019lower}. Therefore, a natural question is whether we can design an algorithm which matches the lower bound for the bias term while simultaneously enjoying the optimal variance decay. In this section, we show that this is possible, and we do so by applying the multistage machinery to the stochastic Optimistic Gradient Descent Ascent (OGDA) algorithm. In this section, we first revisit the existing results on convergence of stochastic OGDA method, and next, show how its multistage version (M-OGDA) can matches both lower bounds simultaneously. 

The stochastic OGDA method is given by:
\begin{align*}
    x_{k+1} &= x_k - 2\alpha \tnabla_x f(x_k,y_k) + \alpha \tnabla_x f(x_{k-1},y_{k-1}) \\ 
    y_{k+1} &= y_k + 2 \alpha \tnabla_y f(x_k,y_k) - \alpha \tnabla_y f(x_{k-1},y_{k-1}) 
\end{align*}
which can also be written as:
\begin{align}
    z_{k+1} = z_k - 2\alpha \tPhi(z_k) + \alpha \tPhi(z_{k-1})
\end{align}
where recall that $z_{k} = (x_k^\top,y_k^\top)^\top$, $\alpha$ is the stepsize, and $\tPhi(\cdot)$ are the stochastic gradients (unbiased estimates of the true gradients) \eqref{stoch_Phi}. 
The OGDA updates have been observed to permorm well empirically for training GANs (see \cite{DBLP:conf/iclr/DaskalakisISZ18}) and have been proved to converge for convex-concave problems (see \cite{mokhtari2019unified, hsieh2019convergence}) which is not true for GDA. 
As shown in \cite{GidelBVVL19, hsieh2019convergence}, the OGDA updates can also be thought of as a single call version of the Extragradient method (since it reuses a gradient from the past) and using this interpretation, the OGDA algorithm can also be written as follows\footnote{In this section, we use the variables w, z to represent the concatenation of the variables $(x,y)$ of the original problem, i.e., $(x^\top, y^\top)^\top$ in order to maintain brevity.}:
\begin{subequations}\label{OGDA_update:main}
\begin{align}
    z_{k+1} &= w_k - \alpha \tPhi(z_k) \label{OGDA_update:a} \\
    w_{k+1} &= w_k - \alpha \tPhi(z_{k+1}) \label{OGDA_update:b}
\end{align}
\end{subequations}
Note that the difference from Extragradient (EG) is that in EG, the update for $z_{k+1}$ involves the gradient at $w_k$ whereas here we use the gradient at $z_k$ instead. We will use this form of the Stochastic OGDA updates for our analysis. From the analysis  \footnote{The analysis in \cite{hsieh2019convergence} is for the case of additive noise. However, it can be easily extended to our setting by conditioning and using the tower rule on the current iterate.} of Theorem $5$ in \cite{hsieh2019convergence}, we have the following result for Stochastic OGDA:
\begin{theorem} (\cite{hsieh2019convergence})\label{Thm_OGDA} 
Suppose that the conditions in Assumptions \ref{var_assmp:main} and \ref{ass:scsc} are satisfied. Let $\{z_k, w_k\}_k$ be the iterates generated by Stochastic OGDA with $0 < \alpha \leq \frac{1}{8L}$. Then, for any $k \geq 1$, we have
\begin{align*}\label{ineq_OGDA_1step}
\E[ & \|w_{k}-z^*\|^2] + \alpha^2 \E[\|\tPhi(z_{k}) - \tPhi(z_{k-1})\|^2] \nonumber \\
&\leq  (1- \alpha \mu) \left(\E[\|w_{k-1}-z^*\|^2]  + \alpha^2 \E[\|\tPhi(z_{k-1}) - \tPhi(z_{k-2})\|^2] \right ) + 6 \alpha^2 \sigma^2.
\end{align*}
\end{theorem}
This is similar to Theorem \ref{Thm_GDA} for GDA. However, we can see that for OGDA, the range of permissible stepsizes goes all the way up to $\mathcal{O}(1/L)$ whereas for GDA, the stepsizes are upper bounded by $\mathcal{O}(\mu/L^2)$.

The result in Theorem \ref{Thm_OGDA} shows that for OGDA with constant stepsize $\alpha$, the iterates converge to an $\bigO(\alpha)$ neighborhood of the saddle point. Next, we analyze a multistage version of OGDA (M-OGDA) similar to the analysis of M-GDA in Section \ref{sec:MGDA}. We show that the iterates of M-OGDA converge to the unique saddle point at a rate where the variance decays as $\mathcal{O}(\sigma^2/n)$, which is optimal (and also achieved by M-GDA), but the bias term decays as $\bigO(1) \exp(-n/\Theta(\kappa))$, which "accelerates" GDA in terms of its dependence on $\kappa$. More formally, we state the following theorem which is analogous to Theorem \ref{main_result} for M-GDA and present the convergence rate of M-OGDA (we omit the proof as it is very similar to that of Theorem \ref{main_result}): 
\begin{algorithm}[t]
     \caption{ Multistage Stochastic Optimistic Gradient Descent Ascent Algorithm (M-OGDA)}
     \label{Algorithm2}
     \begin{algorithmic}
    \STATE {\bfseries Input:}Initial iterate $z_0^0$, the stepsize sequence $\{\al_k\}_{k=1}^K$, the stage-length sequence $\{n_k\}_{k=1}^K$.
    \STATE Set $n_0 = 0$ and $w_0^0 = z_0^0$;
    \FOR{$k = 1;\ k \leq K;\ k = k + 1$}
     \STATE Set $z_0^k = z_{n_{k-1}}^{k-1}$ and $w_0^k = w_{n_{k-1}}^{k-1}$;
    \FOR{$m = 0;\ m < n_k;\ m = m + 1$}
     \STATE Set $z_{m+1}^k = w_m^k - \alpha_k \tPhi(z_m^k)$ 
     \STATE Set $w_{m+1}^k = w_m^k - \alpha_k \tPhi(z_{m+1}^k)$
    \ENDFOR
    \ENDFOR
   \end{algorithmic}
\end{algorithm}
\begin{theorem}\label{main_result_OGDA}
Suppose that the conditions in Assumptions \ref{var_assmp:main} and \ref{ass:scsc} are satisfied. Let $\{(w_m^k,z_m^k)_{m=0}^{n_k}\}_{k=1}^K$ be the iterates generated by M-OGDA (Algorithm \ref{Algorithm2}) with the following parameters
\begin{equation}\label{M-OGDA_parametres}
\begin{aligned}
\alpha_1 &= \frac{1}{8L}, n_1 \geq 1 \\
\alpha_k &= \frac{1}{L 2^{k+3}}, n_k = \ceil{p 2^{k+3} \kappa \log(2)} \quad (k \geq 2),	
\end{aligned}
\end{equation}
where $p \geq 2$ is an arbitrary positive number. Then, for any $k \geq 1$, we have	
\begin{align*}
\E \left [ \| w_{n_k}^k - z^* \|^2 \right ] & \leq \frac{\exp\left (-n_1/(8 \kappa)\right) }{2^{p(k-1)}} \|z_0^0 - z^*\|^2 + \frac{\sigma^2}{2^{k-1} L \mu}.
\end{align*}
\end{theorem}
Similar to the discussion in Section \ref{sec:MGDA}, we next state how our result leads to bounds on distance of each iterate to the saddle point of Problem \ref{main_prob}. In addition, we propose proper choice of parameters in general as well as in the case that the iteration budget is known in advance, the results corresponding to Corollaries \ref{cor_anyN} and \ref{result_general} for M-GDA. Before stating this corollary, we define $\{w_n\}_n$ to be the sequence which is obtained by concatenating the $\{ (w_m^k)_{m=0}^{n_k} \}_{k=1}^K$ sequences, i.e., 
\begin{equation}\label{concatenate_O}
w_n = w_{n-\sum_{i=1}^{k-1} n_i}^{k}  \quad \text{for } \sum_{i=1}^{k-1} n_i < n \leq \sum_{i=1}^{k} n_i.
\end{equation}
\begin{corollary}\label{M-OGDA_AllResults}
Suppose that the conditions in Assumptions \ref{var_assmp:main} and \ref{ass:scsc} are satisfied. Let $\{ (w_m^k,z_m^k)_{m=0}^{n_k} \}_{k=1}^K$ be the iterates generated by M-OGDA (Algorithm \ref{Algorithm2}) with the parameters given in Theorem \ref{main_result}. Also, recall the definition of the concatenated sequence $\{ w_n \}$ from \eqref{concatenate_O}. Then, for any $n > n_1$, we have
\begin{align}
 \E \left [ \| w_n - z^* \|^2 \right]  & \leq \bigO(1) \left ( \frac{\exp\left (-n_1/(8 \kappa)\right) }{\left ((n-n_1)/ (\kappa p) \right )^p} \|z_0 - z^*\|^2 + \frac{p\sigma^2}{(n-n_1) \mu^2} \right ),
\end{align}
In particular, assume choosing $n_1 = \ceil{8 p \kappa \log(p \kappa^2)}$.  Then, for any $n \geq 2 n_1$, we have
\begin{align}
 \E \left [ \| w_n - z^* \|^2 \right]  & \leq \bigO(1) \left ( 1/n^p \|z_0 - z^*\|^2 + \frac{p\sigma^2}{n \mu^2} \right ). 
\end{align}
Also, when the number of iterations $n$ is known in advance, choosing $p=2, n_1= \frac{n}{C}$ with $C \geq 2$, implies
\begin{align}
 \E \left [ \| w_n - z^* \|^2 \right]  & \leq \bigO(1) \left ( \exp\left (-\Theta(1) n/\kappa\right)\|z_0 - z^*\|^2 + \frac{\sigma^2}{n \mu^2} \right )
\end{align}
for any $n \geq 2 \kappa$.
\end{corollary}
Once again, we would like to highlight that the results in Corollary \ref{M-OGDA_AllResults} are presented in terms of the $n^{th}$ iterate $w_n$ which is obtained by concatenating the iterates of all stages, including inner iterations (as given in \eqref{concatenate_O}). As a result, in comparing our results to other methods in Table \ref{table-1}, we take into account the inner stage iterations. 
\section{Conclusion}
In this paper, we propose multistage versions of Gradient Descent Ascent (GDA) and Optimistic Gradient Descent Ascent (OGDA) algorithms to solve the stochastic minimax problems. In particular, these algorithms are the first to achieve linear rate in bias and optimal $\bigO(\sigma^2/n)$ rate in variance, simultaneously. We also show that Multistage OGDA improves the bias rate of Multistage GDA from $\bigO(\exp(-\Theta(1) n/\kappa^2))$ to $\bigO(\exp(-\Theta(1) n/\kappa))$ which is the best known rate in deterministic minimax optimization. 
\newpage
\bibliographystyle{apalike}
\bibliography{main}

\newpage
\appendix

\section{Proof of Lemma \ref{lemma:SCSC_monotone}}\label{lemma:SCSC_monotone_proof}
We have:
\begin{align}
    \langle & \nabla_{x} f(x, y)  - \nabla_{x} f(\hat{x}, \hat{y}), x - \hat{x} \rangle \nonumber \\
    &= \ \langle \nabla_{x} f(x, y), x - \hat{x} \rangle - \langle \nabla_{x} f(\hat{x}, \hat{y}), x - \hat{x} \rangle \nonumber \\
    &\geq f(x, y) - f(\hat{x},y) + \frac{\mu}{2}\|\hat{x} - x \|^2 + f(\hat{x}, \hat{y}) - f(x, \hat{y}) + \frac{\mu}{2}\|\hat{x} - x \|^2 \label{ineq:monotone_1} \\
    &= f(x, y) - f(\hat{x},y) + f(\hat{x}, \hat{y}) - f(x, \hat{y}) + \mu \|\hat{x} - x \|^2
    \label{eq:monotone_1}
\end{align}
where \eqref{ineq:monotone_1} follows from \eqref{SConvex}. Similarly
\begin{align}
    \langle & -\nabla_{y} f(x, y)  + \nabla_{y} f(\hat{x}, \hat{y}), y - \hat{y} \rangle \nonumber \\
    &= \ - \langle \nabla_{y} f(x, y), y - \hat{y} \rangle + \langle \nabla_{y} f(\hat{x}, \hat{y}), y - \hat{y} \rangle \nonumber \\
    &\geq -f(x, y) + f(x,\hat{y}) + \frac{\mu}{2}\|\hat{y} - y \|^2 - f(\hat{x}, \hat{y}) + f(\hat{x}, y) + \frac{\mu}{2}\|\hat{y} - y \|^2 \nonumber \\
    &= -f(x, y) + f(\hat{x},y) - f(\hat{x}, \hat{y}) + f(x, \hat{y}) + \mu \|\hat{y} - y \|^2
    \label{eq:monotone_2}
\end{align}
Adding equations \eqref{eq:monotone_1} and \eqref{eq:monotone_2}, and noting that
\begin{align}
    &\langle \Phi(z)- \Phi(\hat{z}), z - \hat{z} \rangle \nonumber \\
    &= \langle  \nabla_{x} f(x, y)  - \nabla_{x} f(\hat{x}, \hat{y}), x - \hat{x} \rangle + \langle  -\nabla_{y} f(x, y)  + \nabla_{y} f(\hat{x}, \hat{y}), y - \hat{y} \rangle
\end{align}
we 
obtain the right hand side of \eqref{sc_sm_g}. Similarly, we have:
\begin{align}
    \langle & \nabla_{x} f(x, y)  - \nabla_{x} f(\hat{x}, \hat{y}), x - \hat{x} \rangle \nonumber \\
    &= \ \langle \nabla_{x} f(x, y), x - \hat{x} \rangle - \langle \nabla_{x} f(\hat{x}, \hat{y}), x - \hat{x} \rangle \nonumber \\
    &\leq f(x, y) - f(\hat{x},y) + \frac{L}{2}\|\hat{x} - x \|^2 + f(\hat{x}, \hat{y}) - f(x, \hat{y}) + \frac{L}{2}\|\hat{x} - x \|^2 \label{ineq:monotone_11} \\
    &= f(x, y) - f(\hat{x},y) + f(\hat{x}, \hat{y}) - f(x, \hat{y}) + L \|\hat{x} - x \|^2
    \label{eq:monotone_12}
\end{align}
where \eqref{ineq:monotone_11} follows from \eqref{eq:l_smooth_ch}. Similarly
\begin{align}
    \langle & -\nabla_{y} f(x, y)  + \nabla_{y} f(\hat{x}, \hat{y}), y - \hat{y} \rangle \nonumber \\
    &= \ - \langle \nabla_{y} f(x, y), y - \hat{y} \rangle + \langle \nabla_{y} f(\hat{x}, \hat{y}), y - \hat{y} \rangle \nonumber \\
    &\leq -f(x, y) + f(x,\hat{y}) + \frac{L}{2}\|\hat{y} - y \|^2 - f(\hat{x}, \hat{y}) + f(\hat{x}, y) + \frac{L}{2}\|\hat{y} - y \|^2 \nonumber \\
    &= -f(x, y) + f(\hat{x},y) - f(\hat{x}, \hat{y}) + f(x, \hat{y}) + L \|\hat{y} - y \|^2
    \label{eq:monotone_22}
\end{align}
Adding equations \eqref{eq:monotone_12} and \eqref{eq:monotone_22} we 
obtain the left hand side of \eqref{sc_sm_g}.
\section{Proof of Lemma \ref{lem_Phi}} \label{lem_Phi_proof}
From Lemma \ref{lemma:SCSC_monotone}, we have:
\begin{align}
\label{eq:mid_thm_1}
    \langle \Phi(z)- \Phi(\hat{z}), z - \hat{z} \rangle \geq \mu\| z - \hat{z} \|^2
\end{align}
and from Assumption \ref{def:lips_grad}, we have:
\begin{align}\label{eq:mid_thm_2}
    \| \Phi(z) - \Phi(\hat{z}) \|^2 \leq 4L^2 \| z - \hat{z} \|^2.
\end{align}
Combining Equation \eqref{eq:mid_thm_1} and \eqref{eq:mid_thm_1}, we have:
\begin{align*}
    \langle \Phi(z)- \Phi(\hat{z}), z - \hat{z} \rangle \geq \frac{\mu}{4L^2}\| \Phi(z) - \Phi(\hat{z}) \|^2.
\end{align*}
\section{Proof of Theorem \ref{Thm_GDA}}\label{Thm_GDA_proof}
 First, note that for $\rho^2 = 1- \alpha \mu$ and $P = p \otimes I_{m+n}$ with $p=1/\alpha$, we have
\begin{align}\label{p_inequality}
\begin{bmatrix} 
	A^\top P A - \rho^2 P & A^\top P B\\
    B^\top P A & B^\top P B
\end{bmatrix} =
p \begin{bmatrix} 
	1 - \rho^2 & -\alpha\\
    -\alpha & \alpha^2
\end{bmatrix}
= \begin{bmatrix} 
	\mu & -1\\
    -1 & \alpha
\end{bmatrix}
\preceq  
\begin{bmatrix}
    \mu & -1 \\ -1 & \mu/(4L^2)    
\end{bmatrix}
\end{align}
where the last inequality follows from the fact that $\alpha \leq \mu/(4L^2)$. This result implies
\begin{align*}
\begin{bmatrix} 
	z_{k-1}-z^*\\
    \Phi (z_{k-1})
\end{bmatrix}^\top
\begin{bmatrix} 
	A^\top P A - \rho^2 P & A^\top P B\\
    B^\top P A & B^\top P B
\end{bmatrix}
\begin{bmatrix} 
	z_{k-1}-z^*\\
    \Phi (z_{k-1})
\end{bmatrix} \leq 
\begin{bmatrix} 
	z_{k-1}-z^*\\
    \Phi (z_{k-1})
\end{bmatrix}^\top
\begin{bmatrix}
    \mu & -1 \\ -1 & \mu/(4L^2)    
\end{bmatrix}
\begin{bmatrix} 
	z_{k-1}-z^*\\
    \Phi (z_{k-1})
\end{bmatrix}.
\end{align*}
Substituting left and right hand side by using Corollary \ref{corollary:Matrix_format} and Lemma \ref{storage_update}, respectively, yields 
\begin{align}
\E[V_p(z_{k})] - (1-\alpha \mu) \E[V_p(z_{k-1})] - 2 \sigma^2 \alpha^2 p \leq 0.    
\end{align}
Finally, dividing both sides by $p$ completes the proof \eqref{ineq_GDA_1step}. To show the second result, note that by using \eqref{ineq_GDA_1step} sequentially, we have
\begin{align}
\E[\|z_{k}-z^*\|^2] & \leq  (1- \alpha \mu)^{k} \E[\|z_{0}-z^*\|^2] + 2 \alpha^2 \sigma^2 \left (1+ (1-\alpha \mu) + ... + (1-\alpha \mu)^k \right ) \nonumber \\
& \leq  (1- \alpha \mu)^{k} \E[\|z_{0}-z^*\|^2] + 2 \alpha^2 \sigma^2 \sum_{i=0}^\infty (1-\alpha \mu)^i \nonumber \\
& =  (1- \alpha \mu)^{k} \E[\|z_{0}-z^*\|^2] + 2 \alpha^2 \sigma^2 \frac{1}{1-(1-\alpha \mu)} = (1- \alpha \mu)^{k} \E[\|z_{0}-z^*\|^2] + 2 \alpha \sigma^2/\mu \nonumber
\end{align}
which completes the proof.
\section{Proof of Example \ref{lemma:quad_example}}\label{lemma:quad_example_proof}
We have the function:
\begin{align}
    f(x,y) = \frac{\mu}{2} x^2 + L xy - \frac{\mu}{2}y^2
\end{align}
The gradient at step $k$ is corrupted by noise $\xi_k^x$ and $\xi_k^y$ which we assume to be iid $\sim \mathcal{N}(0,\sigma)$. The GDA method when applied to this problems leads to:
\begin{align}
    x_{k+1} &= x_k - \alpha (\mu x_k + Ly_k + \xi^x_k) \nonumber \\
    y_{k+1} &= y_k + \alpha (-\mu y_k + Lx_k + \xi^y_k)
\end{align}
This gives:
\begin{align}
    \mathbb{E}\left[ \|x_{k+1}\|^2 + \|y_{k+1}\|^2 \right] &= \left((1 - \alpha \mu)^2 + \alpha^2L^2 \right)(\mathbb{E}\left[ \|x_{k}\|^2 + \|y_{k}\|^2 \right]) + 2\alpha^2 \sigma^2 \nonumber \\
    &= \left(1 - 2\alpha \mu + \alpha^2(\mu^2+L^2) \right)(\mathbb{E}\left[ \|x_{k}\|^2 + \|y_{k}\|^2 \right]) + 2\alpha^2 \sigma^2 \nonumber \\
    &\geq \left(1 - 2\alpha \mu \right)(\mathbb{E}\left[ \|x_{k}\|^2 + \|y_{k}\|^2 \right]) + 2\alpha^2 \sigma^2
    \label{eq:quad_mid}
\end{align}
which proves the first part of the lemma. Note that when the gradients are not corrupted by noise (i.e. when $\sigma = 0$, we have)
\begin{align}
    \|x_{k+1}\|^2 + \|y_{k+1}\|^2 &= \left((1 - \alpha \mu)^2 + \alpha^2L^2 \right)( \|x_{k}\|^2 + \|y_{k}\|^2 ) \nonumber \\
    &= \left(1 - 2\alpha \mu + \alpha^2(\mu^2+L^2) \right)( \|x_{k}\|^2 + \|y_{k}\|^2 )
\end{align}

The coefficient on the right side is minimized for $\alpha = \frac{\mu}{\mu^2 + L^2}$. Substituting this in equation \eqref{eq:quad_mid}, we get:
\begin{align}
     \|x_{k+1}\|^2 + \|y_{k+1}\|^2   &= \left(1 - \frac{2\mu^2}{\mu^2 + L^2} + \frac{\mu^2}{\mu^2 + L^2} \right) (\|x_{k}\|^2 + \|y_{k}\|^2) \nonumber \\
    &= \left(1 - \frac{\mu^2}{\mu^2 + L^2} \right)( \|x_{k}\|^2 + \|y_{k}\|^2 ) 
\end{align}
Now, making the substitution $\kappa = \frac{L}{\mu}$, we have:
\begin{align}
     \|x_{k+1}\|^2 + \|y_{k+1}\|^2   &= \left(1 - \frac{1}{1 + \kappa^2} \right)(\|x_{k}\|^2 + \|y_{k}\|^2 )  \nonumber \\
    &\geq  \left(1 - \frac{1}{\kappa^2} \right)(\|x_{k}\|^2 + \|y_{k}\|^2 ) 
\end{align}
Therefore, for any other stepsize $\alpha > 0$, we have:
\begin{align}
     \|x_{k+1}\|^2 + \|y_{k+1}\|^2  &\geq  \left(1 - \frac{1}{\kappa^2} \right)(\|x_{k}\|^2 + \|y_{k}\|^2)
\end{align}
\section{Proof of Corollary \ref{cor_anyN}}\label{Proof_cor_anyN}
Let us define $T(k):= \sum_{i=1}^{k} n_i$. 
Note that, for $k \geq 2$, the fact that $\ceil{x} \leq 2x$ for positive $x$ implies:
\begin{align}
T(k) - n_1 = \sum_{i=2}^{k} n_i = \sum_{i=2}^{k} \ceil{p 2^{i+2} \kappa^2 \log(2)} \leq 2 p \kappa^2 \log(2) \sum_{i=2}^{k} 2^{i+2}.	
\end{align}
As a consequence, and using $\sum_{i=2}^{k} 2^{i+2} = 16 (2^{k-1}-1)$, we have
\begin{equation}\label{bound1_n-n_1}
T(k) - n_1 \leq 32 p \kappa^2 \log(2) (2^{k-1}-1).
\end{equation}
Now, let $k$ be the largest number such that $T(k) < n$, i.e., $T(k) < n \leq T(k+1)$. Thus, using \eqref{bound1_n-n_1}, we obtain 
\begin{equation}\label{bound2_n-n_1}
n-n_1 \leq T(k+1) -n_1 \leq 32 p \kappa^2 \log(2) (2^{k}-1),
\end{equation}
and as a result, we have
\begin{equation}\label{bound_n-n_1}
2^k \geq \Theta(1) \frac{n-n_1}{p \kappa^2}	
\end{equation}
 where the constants in $\Theta(1)$ are independent of problems' parameters. 

Next, note that, by Theorem \ref{main_result}, we have
\begin{align}\label{ineq_recall}
\E \left [ \| z_{n_k}^k - z^* \|^2 \right ] & \leq \frac{\exp\left (-n_1/(4 \kappa^2)\right) }{2^{p(k-1)}} \|z_0 - z^*\|^2 + \frac{\sigma^2}{2^k L^2}	.
\end{align}
Also, by Theorem \ref{Thm_GDA} for stage $k+1$, we have
\begin{align}
\E \left \|z_n - z^*\|^2 \right ] &\leq \exp(- \alpha_{k+1} (n-T(k)) \mu) \E \left [ \| z_{n_k}^k - z^* \|^2 \right ] + 2 \alpha_{k+1} \sigma^2/\mu \\
& \leq 	\E \left [ \| z_{n_k}^k - z^* \|^2 \right ] + \frac{\sigma^2}{2^{k+2} L^2} \\
& \leq \frac{\exp\left (-n_1/(4 \kappa^2)\right) }{2^{p(k-1)}} \|z_0 - z^*\|^2 + 2 \frac{\sigma^2}{2^k L^2}
\end{align}
where the last inequality follows from \eqref{ineq_recall}. Now, plugging in \eqref{bound_n-n_1} in this bound completes the proof.
\end{document}